\documentclass[final]{siamltex}

\usepackage{amsmath,amssymb}
\usepackage{graphicx,epsfig,color}
\usepackage{enumerate}

\newcommand{\f}{\frac}

\newcommand{\R}{\mathbb {R}}

\newcommand{\mc}{\mathcal}

\title{Seasonal influence on age-structured invasive species with yearly generation}

\author{Yingli Pan\thanks{Department of Mathematics, Harbin Institute of Technology, Harbin, Heilongjiang, 150001, China ({\tt ylpanhit@gmail.com})} \and Jian Fang\thanks{Institute for Advanced Studies in Mathematics and Department of Mathematics, Harbin Institute of Technology, Harbin, Heilongjiang, 150001, China ({\tt jfang@hit.edu.cn})}  \and  Junjie Wei\thanks{School of Science, Harbin Institute of Technology in Weihai, Weihai, Shandong, 264209, China ({\tt weijj@hit.edu.cn})}}
\begin{document}

\maketitle

\begin{abstract}
How do seasonal successions influence the propagation dynamics of an age-structured invasive species? We investigate this problem by considering the scenario that the offsprings are reproduced in spring and then reach maturation in fall within the same year. For this purpose, a reaction-diffusion system is proposed, with yearly periodic time delay and spatially nonlocal response caused by the periodic developmental process.  By appealing to the recently developed dynamical system theories, we obtain the invasion speed $c^*$ and its coincidence with the minimal speed of time periodic traveling waves. {\it The characterizations of $c^*$ suggest that (i) time delay decreases the speed and its periodicity may further do so; (ii) the optimal time to slow down the invasion is the season without juveniles; (iii) the speed increases to infinity with the same order as the square root of the diffusion rate.}
\end{abstract}

\begin{keywords}
Periodic delay; Seasonal influence; Invasion speed; Traveling wave.
\end{keywords}

\begin{AMS}
Primary: 92D25; Secondary:  34K13, 35C07, 37N25
\end{AMS}

\pagestyle{myheadings}
\thispagestyle{plain}
%\markboth{TEX PRODUCTION AND V. A. U. THORS}{SIAM MACRO EXAMPLES}

\section{Introduction}
Seasonal successions bring invasive species temporally dynamic habitats, which then affect their life cycles, including breeding, development, mobility, maturation, mortality, etc, and further give rise to the propagation complexity during their invasion process. How seasonality influences the propagation of an age-structured invasive species is a challenging problem from the viewpoints of both mathematical modeling and analysis. 

In this paper, we intend to investigate such a problem for a single invasive species by considering the scenario that the species has distinct breeding and maturation seasons. More precisely, we assume that the species has the following biological characteristics.

\begin{enumerate}
  \item[(B1)]
     The species can be classified into two stages by age: mature and immature. An individual at time $t$ belongs to the mature class if and only if its age exceeds the time dependent positive number $\tau(t)$. Within each stage, all individuals share the same behavior. 
  \item[(B2)]
     Adults reproduce offsprings once a year in a fixed season (spring), and they reach maturation in another season (fall) within the same year.
   \item[(B3)]All behaviors are yearly periodic due to seasonal succession.
          \item[(B4)]
     The spatial habitat is ideally assumed to be one dimensional and homogeneous in locations.
\end{enumerate}
We refer to Fig. 1.1 for a schematic illustration of the life cycle in a year $[0,T]$. Since the habitat is dynamic in time, the biological developmental rate of species may be different from one time to another. So we assume that the duration $\tau=\tau(t)$ from newborn to being adult is yearly time periodic. Further, since the developmental rate depends only on time, juveniles cannot reach maturation before those born ahead of them. Therefore, one has following additional implicit assumption.
\begin{enumerate}
\item[(B5)]  $t-\tau(t)$ is strictly increasing in $t$, that is, $\tau'(t)< 1$ if $\tau$ is a smooth function.
\end{enumerate}

\begin{figure}[t]\label{f1}
  \begin{center}
  \includegraphics[width=.99\textwidth]{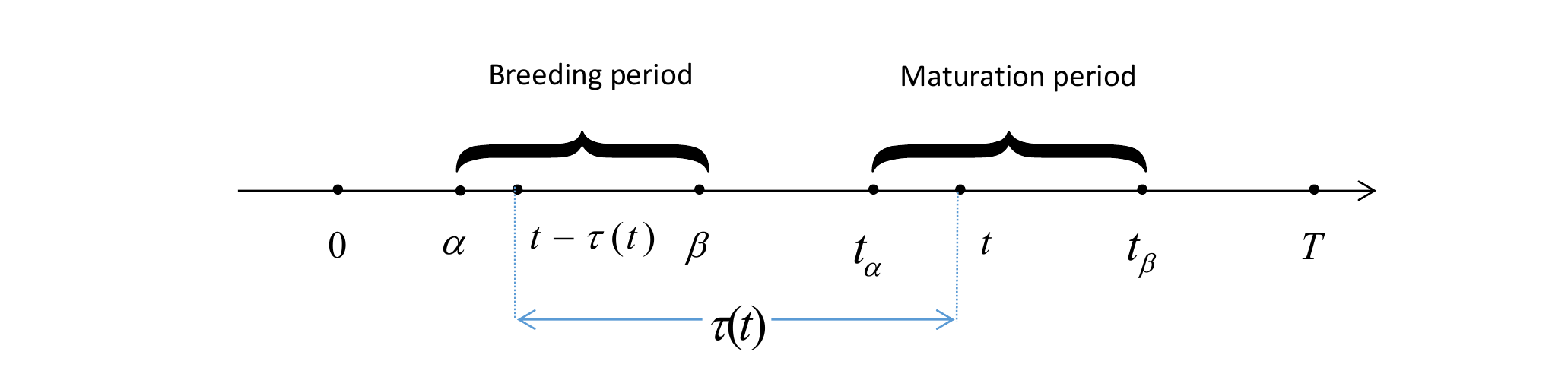}
    \caption{Schematic illustrations of the life cycle for the species with yearly generation, where $[\alpha,\beta]$ is the breeding period  and  $[t_\alpha,t_\beta]$ is the maturation period.}
  \end{center}
\end{figure}

To mathematically model such an invasive species, we start from the following growth law for two stages of age-structured population (see Metz and Diekmann \cite{MetzDiekmann86}):
\begin{equation}\label{growthlaw}
  \left\{
  \begin{aligned}
&\left( \frac{\partial }{\partial t}+\frac{\partial }{\partial a}\right) p=D_M(t)\frac{\partial^{2}}{\partial x^{2}}p-d_M(t)p,  \quad  a\ge \tau(t)\\
&\left( \frac{\partial }{\partial t}+\frac{\partial }{\partial a}\right) p=D_I(t)\frac{\partial^{2}}{\partial x^{2}}p-d_I(t)p,  \quad 0<a< \tau(t),
   \end{aligned}
   \quad t>0, x\in\R,
  \right.
\end{equation}
where $p=p(x,t,a)$ denotes the density of species of age $a$ at time $t$ and location $x$, $D_M,D_I$ are the diffusion rates, and $d_M, d_I$ are the death rates. Clearly, the total mature population $u$ and immature population $v$ at time $t$ and location $x$ can be represented, respectively, by the integrals
\[
u(t,x)=\int_{\tau(t)}^{+\infty} p(x,t,a)da, \quad v(t,x)=\int_0^{\tau(t)} p(x,t,a) da.
\]
Keeping a smooth flow for the paper, we write down here the model while leaving in the next section the derivation details, which are highly motivated by the ideas behind a few recent works. We will explain the details in the next few paragraphs. 
\begin{equation}\label{Model}
  \left\{
  \begin{aligned}
    &\frac{\partial u}{\partial t}=D_M(t)\frac{\partial^{2} u}{\partial x^{2}}-d_M(t)u+R(t,u(t-\tau(t), \cdot))(x),\\
    &\frac{\partial v}{\partial t}=D_{I}(t)\frac{\partial^{2} v}{\partial x^{2}}-d_{I}(t)v+b(t,u(x,t))-R(t,u(t-\tau(t), \cdot))(x),
  \end{aligned}
   \quad t>0, x\in\mathbb{R},
  \right.
\end{equation}
where $R$ in general is a nonlocal term, meaning the recruitment  rate of mature population at time $t$ and location $x$, and function $b$ is the birth rate.  The recruitment term $R$ turns out to have the following expression
\begin{equation}\label{recruitment}
R(t, \phi)(x)=(1-\tau'(t)) b(t-\tau(t),\phi)\ast k_I(t, t-\tau(t),\cdot)(x),
\end{equation}
where
\begin{equation}
k_I(t,t-\tau(t),x)=\frac{\varepsilon(t)}{\sqrt{4\pi \sigma(t)}}\exp\left\{ -\f{x^2}{4\sigma(t)}\right\}=:\varepsilon(t)f_{\sigma(t)}(x)
\end{equation}
is the Green function of $\partial_t \rho=D_I(t)\partial_{xx}-d_I(t), x\in\R$. Assumption (B5), i.e., $\tau'(t)< 1$, ensures the positivity and well-posedness of the recruitment term, in which, $(1-\tau'(t))\varepsilon(t)$ is the survival rate from newborn to being adult, $f_{\sigma(t)}$ is the redistribution kernel in space for newborns upon they become mature, and $\sigma(t)$, derived from the diffusion of immature population, measures how strong the nonlocal interaction (induced by the diffusibility of immature population) is. 

Let us recall some related works that motivate our study. In a pioneer paper on periodic delays, Freedman and Wu \cite{FreedmanWu92} studied the existence of a periodic solution of a single species population model 
\begin{equation}
u'(t)=u(t)[a(t)-b(t)u(t)+c(t)u(t-\tau(t))],
\end{equation}
where the positive functions $a,b,c,\tau$ are all $\omega$-periodic for some $\omega>0$.  Li and Kuang \cite{LiKuang01} investigated a two species Lotka-Volterra model with distributed periodic delays. Recently, an increasing attention has been paid to the emerging term $1-\tau'(t)$ in modeling the periodic/state-dependent development for juveniles. We refer to Barbarossa et al. \cite{BarbarossaEtal14}, Wu et al. \cite{WuEtal15} and Kloosterman et al. \cite{KloostermanEtal16} for the biological meanings in various scenarios.  Very recently,  the basic reproduction number theory has been developed by Lou and Zhao \cite{LouZhao17} for a large class of disease models with periodic delays, by employing which Wang and Zhao \cite{WangZhao17} analyzed a malaria model with temperature-dependent incubation period and Liu et al. \cite{LiuLouWu17} analyzed a tick population model subject to seasonal effects. 

By considering a single invasive species with two stages distinguished by age, So et al. \cite{SoWuZou01} formulated a reaction-diffusion model in $\R$ with constant time delay and spatially nonlocal interaction
\begin{equation}\label{model-sowuzou}
\partial_t u=\partial_{xx} u-d u+e^{-\gamma \tau}b(u(t-\tau,\cdot))\ast k_{\sigma(\tau)}, \quad x\in\R,
\end{equation}
where $u$ represents the density of mature population, the positive constant $\tau$ is the maturation period,  and the whole nonlocal term is the recruitment of mature population, accounting for the redistribution of survived juveniles at the time when they reach maturation. Such a mechanism generating nonlocal term has stimulated many developments in differential equations, dynamical systems and nonlinear analysis \cite{GourleyWu06}.  Under suitable conditions, it has been shown in \cite{LiangZhao07} that the invasion speed $c^*$ coincides with the minimal speed of traveling waves that is determined by the variational formula $c^*=\inf_{\mu>0} \f{\lambda(\mu)}{\mu}$, where $\lambda(\mu)$ is the principal eigenvalue of
\begin{equation}
\lambda=\mu^2-d+e^{-\gamma \tau}b'(0) k_{\sigma(\tau)} \ast e^{-\lambda (\tau+\cdot)}.
\end{equation}
It is proved by Li et al. \cite{LiRuanWang07} that $c^*$ is decreasing in delay $\tau$ (see also \cite{WangLiRuan08}).  

With the aforementioned works, a natural question then arises: {\it what if the periodic development rate is incorporated into the population model \eqref{model-sowuzou}? Further, how does the periodicity affect the invasion speed $c^*$?} For this purpose, we proposed the model \eqref{Model}, which is a generalized version of \eqref{model-sowuzou}.

If assuming all time periodic parameters are constants in \eqref{Model}, then one may retrieve \eqref{model-sowuzou}. If assuming only delay is a constant, then one has the model studied by Jin and Zhao \cite{JinZhao08}. If assuming that the mobility of immature population can be ignored, i.e., passing $\sigma(t)\to 0$, then one can see that $f_{\sigma(t)}$ tends to the Dirac measure and the nonlocal term $R$ reduces to the local form 
\begin{equation}
(1-\tau'(t))\varepsilon(t)b(t-\tau(t),u(t-\tau(t),x)).
\end{equation}
If assuming the mobility of mature population can be ignored, i.e., $D_M(t)\equiv 0$, then one has the following integro-differential equation for the mature population:
\begin{equation}
\frac{\partial u(t,x)}{\partial t}=-d_M(t)u(t,x)+R(t,u(t-\tau(t), \cdot))(x),
\end{equation}
which looks simpler than the diffusive equation but has more mathematical challenges due to the lack of regularity in spatial variable $x$.

Clearly, model \eqref{Model} is partially uncoupled. But it is still hard to fully understand how the time heterogeneity influences the propagation dynamics, especially the periodicity of time delay. In virtue of biological assumption (B2), we may cast the first equation of model \eqref{Model} for the mature population into an iterative system determined by its Poincar\'e map. This feature will be essentially of help for the mathematical analysis. As for the immature population, the second equation in model \eqref{Model} can be regarded as a linear time periodic reaction-diffusion equation with an inhomogeneous reaction term once the invasion of mature population is understood. To analyze such an equation, the following conservation equality will play an vital role. It comes from the biological observation that all newborns will become mature in the same year as when they were born.

\begin{equation}\label{conservation}
\int_0^T k_I(t,s,\cdot)\ast b(s,u(s,\cdot))ds=\int_{0}^T k_I(t,s,\cdot)\ast R(s,u(s-\tau(s),\cdot))ds, \quad t>T,
\end{equation}
where $k_I(t,s,x)$ is the Green function of $\partial_t \rho=D_I(t)\partial_{xx}\rho-d_I(t)\rho$.

The rest of this paper is organized as follows. In section 2, we formulate the model \eqref{Model} and reduce its first equation to the iterative system \eqref{EqQ=S+R}-\eqref{R}. In section 3, we investigate the global dynamics for the spatially homogeneous system of   \eqref{EqQ=S+R}-\eqref{R}. In section 4,  we first employ the dynamical system theories in \cite{FangZhao14, LiangYiZhao06, LiangZhao07, Weinberger82} to obtain the invasion speed $c^*$ and its coincidence with the minimal speed of traveling waves, and then we investigate the parameter influence on $c^*$, including the time periodicity. In the last section, we are back to model \eqref{Model} and establish the propagation dynamics with the help of \eqref{conservation}.

\section{Model}
In this section, we first formulate the reaction-diffusion model system \eqref{Model} from the growth law  \eqref{growthlaw} and then reduce it to the iterative system  \eqref{EqQ=S+R}-\eqref{R} by using the biological characteristics as assumed in the previous section. 

\subsection{Reaction-diffusion model with periodic delay}

Let $p(t,x,a)$ denote the population density of the species under consideration at time $t\geq0$, age $a\geq0$ and location $x\in\R$. According to assumption (B1), the total matured population at time $t$ and location $x$ is given by
\begin{equation}\label{mp}
    u(t,x)=\int_{\tau(t)}^{+\infty}p(t,x,a)da.
\end{equation}
It is natural to assume that $p(t,x,+\infty)=0$. Differentiating the both sides of \eqref{mp} in time yields
\begin{equation}\label{diffmp}
\f{\partial}{\partial t} u(t,x)=\int_{\tau(t)}^\infty \f{\partial}{\partial t} p(t,x,a)da-\tau'(t)p(t,x,\tau(t)).
\end{equation}
By using the growth law \eqref{growthlaw}, we obtain that
\begin{eqnarray}
\int_{\tau(t)}^\infty \f{\partial}{\partial t} p(t,x,a)da &&= \int_{\tau(t)}^\infty \left[-\f{\partial}{\partial a}+D_M(t)\f{\partial^2}{\partial x^2}-d_M(t)\right] p(t,x,a)da\nonumber\\
&&=D_M(t) \f{\partial^2}{\partial x^2} u(t,x)-d_M(t)u(t,x)+p(t,x,\tau(t)).
\end{eqnarray}
Consequently, \eqref{diffmp} becomes
\begin{equation}
\f{\partial}{\partial t} u(t,x)=D_M(t) \f{\partial^2}{\partial x^2} u(t,x)-d_M(t)u(t,x)+ (1-\tau'(t))p(t,x,\tau(t)).
\end{equation}
To obtain a closed form of the model, one needs to express $p(t,x,\tau(t))$ by $u$ in a certain way. Indeed, $p(t,x,\tau(t))$ represents the newly matured population at time $t$, and it is the evolution result of newborns at $t-\tau(t)$. That is, there is an evolution relation between the quantities $p(t,x,\tau(t))$ and $p(t-\tau(t),x,0)$. Such a relation is obeyed by the second equation of the growth law \eqref{growthlaw}. More precisely, the relation is the time-$\tau(t)$ solution map of the following evolution equation
\begin{equation}\label{evolution}
 \left\{
  \begin{aligned}
    &\frac{\partial q}{\partial s}=D_{I}(t-\tau(t)+s)\frac{\partial^{2}q}{\partial x^{2}}-d_{I}(t-\tau(t)+s)q,\quad x\in\mathbb{R},\ 0\leq s\leq \tau(t),\\
    &q(0,x)=p(t-\tau(t),x,0),\quad x\in\mathbb{R}.
  \end{aligned}
  \right.
\end{equation}
And the newborns $p(t-\tau(t),x,0)$ is given by then birth $b(t-\tau(t),u(x,t-\tau(t)))$. Next, by solving the linear  Cauchy problem \eqref{evolution} we obtain
\begin{equation}\label{Rform}
p(t,x,\tau(t))=q(\tau(t),x)=\f{R(t,u(t-\tau(t),\cdot))(x)}{1-\tau'(t)},
\end{equation}
where $R$ is defined in \eqref{recruitment}. Combining \eqref{diffmp} and \eqref{Rform}, we arrive at a reaction-diffusion equation for mature population $u$, that is, the first equation in the model \eqref{Model}.

As for the total immature population, it is defined by $v(t,x):=\int_0^{\tau(t)}p(t,x,a)da$. Then one may employ the same idea as above to derive another equation. Indeed, differentiating $v$ in time yields
\begin{equation}
\f{\partial}{\partial t}v(t,x)=\int_0^{\tau(t)} \f{\partial}{\partial t}p(t,x,a)da +\tau'(t)p(t,x,\tau(t)),
\end{equation}
which, thanks to the second equation of growth law \eqref{growthlaw}, becomes
\begin{equation}
\f{\partial}{\partial t} v(t,x)=D_I(t)\f{\partial^2}{\partial x^2} v(t,x)-d_Iv(t,x)-(1-\tau'(t))p(t,x,\tau(t)),
\end{equation}
which, together with \eqref{Rform}, gives rise to the second equation of model \eqref{Model} for the immature population.

\subsection{Reduction to a mapping}

Let $T$($=$ a year) be the period of all time dependent functions. Assume that $[\alpha,\beta]\subset (0,T)$ is the time duration of breeding season and $[t_\alpha,t_\beta]\subset (0,T)$ is the time duration of maturation season.  As assumed in (B2), the species has distinct breeding and maturation seasons, we then have the following relation
\begin{equation}
  [\alpha,\beta]\cap [t_{\alpha}, t_{\beta}]=\emptyset.
\end{equation}
Further, $t_{\alpha}$ and $t_{\beta}$ satisfy
\begin{equation}
t_{\alpha}-\tau(t_{\alpha})=\alpha,\quad t_{\beta}-\tau(t_{\beta})=\beta.
\end{equation}
Biologically it means that within a year all newborns are given in $[\alpha,\beta]$ and then they reach maturation in a coming season $[t_\alpha,t_\beta]$ within the same year. We refer to Fig. 1.1 for a schematic illustrations of the life cycle for this particular species. With this in mind, we infer that for $t\in[0,T]$,
\begin{equation}
\text{$b(t,\cdot)\equiv 0$ when $t\not\in[\alpha,\beta]$; $\quad$ $R(t,\cdot)\equiv 0$ when $t\not\in [t_\alpha, t_\beta]$.}
\end{equation}
This feature helps us to further infer that in the end of $n$-th year the mature population size $u(nT,x)$ only depends on its initial size $u((n-1)T,x)$ at the beginning of that year, that is, there is a relation $Q$ mapping $u((n-1)T,x)$  to $u(nT,x)$ for $n\ge 1$.

Next, we figure out the expression of $Q$ from an evolution viewpoint. During the year $[0,T]$, the mature population experience only natural death and the recruitment (with diffusion all the time), so we can decompose $Q$ into two parts:
\begin{equation}\label{EqQ=S+R}
Q[\varphi]=\mc{S}[\varphi]+\mc{R}[\varphi],
\end{equation}
where $\mc{S}$ means the survival part of initial value after a period of evolution and $\mc{R}$ is the new contribution at the end of this year by the next generation. Thus,
\begin{equation}\label{S}
\mc{S}[\varphi]=k_M(T,0,\cdot)\ast\varphi
\end{equation}
and
\begin{equation}\label{R}
\mc{R}[\varphi]=\int_{t_\alpha}^{t_\beta}k_M(T,s,\cdot)\ast R(s,k_M(s-\tau(s),0,\cdot)\ast \varphi) ds,
\end{equation}
where $k_M(t,s,x)$ is the Green function of $\partial_t\rho=D_M(t)\partial_{xx}\rho-d_M(t)\rho$. So far, we have got the complete form of $Q$. An alternative way to obtain the expression of $Q$ is to solve the first equation of \eqref{Model}.

The iterative system $\{Q^n\}_{n\ge 0}$ will be sufficient to determine the propagation dynamics of the mature population.

\section{Dynamics of the spatially homogeneous map}

Restricting $Q$ on $\R$ we obtain the map $\overline{Q}: \R\to\R$ defined by
\begin{equation}
\overline{Q}[z]=z\overline{k}_M(T,0)+\int_{t_\alpha}^{t_\beta} \overline{k}_M(T,s)R(s,z\overline{k}_M(s-\tau(s),0)ds,
\end{equation}
where $\overline{k}_M(t,s)=e^{-\int_s^t d_M(\omega)d\omega}$. Before analyzing the map $\bar{Q}$, let us first recall the mathematical assumptions implied by the biological concerns.
\begin{enumerate}
\item[(A1)] (Seasonality) $D_M\ge 0, D_I\ge 0, d_M>0, d_I>0,\tau>0, b\ge 0$ are all $C^1$ functions and $T-$periodic in time.
\item[(A2)] (Distinct breading and maturation seasons)  Assume that 
\begin{equation}
0 < \alpha \leq \beta < t_{\alpha} \leq t_{\beta} < T,
\end{equation}
where $t_{\alpha},t_{\beta}$ satisfy
\begin{equation}
t_{\alpha}-\tau(t_{\alpha})=\alpha,\quad t_{\beta}-\tau(t_{\beta})=\beta.
\end{equation}
Further, we assume that $b(t,u)=p(t)h(u)$, where $p\geq 0, h\geq 0$ and $p(t)\equiv 0$ for $t\in [0,\alpha]\cup [\beta,T]$.  
\item[(A3)](Ordering in maturation) $\tau'(t)<1, t\in\R$.
\end{enumerate}
We further assume that
\begin{enumerate}
\item[(A4)] (Unimodality) Assume that $h\in C^1$ with $h(0)=0=h(+\infty)$ and there exists $z^*>0$ such that $h(z)$ is increasing for $z\in [0, z^*)$ and decreasing for $z\in[z^*,+\infty)$. 
\item[(A5)](Sublinearity) $h(\lambda z)\ge \lambda h(z)$ for $z\ge 0$ and $\lambda\in (0,1)$.
\end{enumerate}
The Ricker type function $pze^{-qz}$ is a typical example of $h$ satisfying (A4) and (A5).

Define
\begin{equation}\label{delta}
L:= \f{\overline{k}_{M}(T,0)}{1-\overline{k}_{M}(T,0)}\int_{t_\alpha}^{t_\beta}\f{\partial_{\varphi} R(s,0)}{\overline{k}_M(s,s-\tau(s))}ds,
\end{equation}
where $R$ is defined in \eqref{Rform}. By a computation we have
\begin{equation}
\partial_{\varphi} R(s,0)=(1-\tau'(s))\varepsilon(s) p(s-\tau(s))h'(0).
\end{equation}

We are now ready to present the following threshold dynamics for the map $\overline{Q}:\R\to\R$.

\begin{theorem}\label{kinetic-dynamics-thm}
Assume that (A1)-(A5) hold. Then the following statements are valid:
 \begin{enumerate}
  \item[(i)] If $L>1$, then $\overline{Q}$ admits at least one positive fixed point. Denote the minimal one by $u^*$. Then $\lim_{n\to\infty} \overline{Q}^n[u]=u^*$ provided that $u\in(0, u^*]$ and $u^*\overline{k}_M(\alpha,0)\le z^*$, where $z^*$ is defined in (A4).
  \item[(ii)] If $L<1$, then $\lim_{n\to\infty}\overline{Q}^n[u]=0$ for $u\ge 0$.
   \end{enumerate}
\end{theorem}
\begin{proof}
(i) By the expression of $R$, one has
\begin{equation}\label{comp1}
\lim_{z\to 0} \f{R(s,z\overline{k}_M(s-\tau(s),0)}{z\overline{k}_M(s-\tau(s))}=\partial_{\varphi} R(s,0) \quad \text{uniformly in $s\in [0,T]$.}
\end{equation}
Note that
\begin{eqnarray}\label{comp2}
&&\f{1}{z\overline{k}_M(T,0)}(\overline{Q}[z]-z)\nonumber\\ =&&1-\f{1}{\overline{k}_M(T,0)}+\int_{t_\alpha}^{t_\beta} \f{1}{\overline{k}_M(s,s-\tau(s))}\f{R(s,z\overline{k}_M(s-\tau(s),0))}{z\overline{k}_M(s-\tau(s),0)}ds\nonumber\\
\rightarrow&&
\begin{cases}
\left[1-\f{1}{\overline{k}_M(T,0)}\right] (1-L)>0, & \text{as $z\to0$, due to $L>1$,}\\
1-\f{1}{\overline{k}_M(T,0)}<0, & \text{as $z\to +\infty$}.
\end{cases}
\end{eqnarray}
Then the existence of positive fixed point follows from the intermediate value theorem. Since $L>1$, we have the minimal positive fixed point $u^*$. If additionally $u^*\overline{k}_M(\alpha, 0)\le z^*$, then one may check that $\overline{Q}[z]$ is non-decreasing in $z\in [0,u^*]$. Since $\overline{Q}[z]>z$ for $z\in(0,u^*)$, we have $\lim_{n\to\infty}\overline{Q}^n[z]$ exists. Clearly, the limit $z_\infty$ is a fixed point in $(0,u^*]$. So $z_\infty$ must be $u^*$.

(ii) From the computations \eqref{comp1} and \eqref{comp2}, we infer that
\begin{equation}
\f{1}{z\overline{k}_M(T,0)}(\overline{Q}[z]-z)\le \left[1-\f{1}{\overline{k}_M(T,0)}\right] (1-L)<0, \quad L<1.
\end{equation}
As such, there exists $\delta\in (0,1)$ such that $\overline{Q}[z]\le (1-\delta)z$ for $z>0$. Hence, $\overline{Q}^n[z]\le (1-\delta)^n z$, which converges to $0$ for $z>0$.
\end{proof}

\section{Propagation dynamics of the iterative system $\{Q^n\}_{n\ge 0}$}

Recall that
\begin{equation}
Q[\varphi]=k_M(T,0,\cdot)\ast\varphi+\int_{t_\alpha}^{t_\beta}k_M(T,s,\cdot)\ast R(s,k_M(s-\tau(s),0,\cdot)\ast \varphi )ds,
\end{equation}
where $k_M(t,s,x)$ is the Green function for the heat equation $\partial_t \rho=D_M(t)\partial_{xx}\rho-d_M(t)\rho$. Clearly, if $D_M\equiv 0$, i.e., the mature population does not move, then $k_M(t,s,x)=e^{-\int_s^t d_M(\omega)d\omega}$. In such a case, $Q$ is not compact. Thus, to include the case $D_M\equiv 0$ we will not assume any compactness for $Q$.

In this section, we shall first apply the dynamical system theory in \cite{Weinberger82, LiangZhao07} to establish the existence of invasion speed as well as its variational characterization by related eigenvalue problems. And then, we apply the results in \cite{FangZhao14} to obtain the existence of the minimal wave speed and its coincidence with the stabled invasion speed. To apply these theories, one needs to choose appropriate phase spaces. More precisely,  we will work in the continuous function space for the spreading speed and monotone function space for traveling waves.

\subsection{Existence of spreading speed $c^*$ and its coincidence with the minimal wave speed}
Let $\mathcal{C}:= BC(\mathbb{R};\mathbb{R})$, consisting of all bounded continuous functions from $\R$ to $\R$,  be endowed with the compact open topology, which can be induced by the following norm
\begin{equation}\label{EqNorm}
  \|u\|= \sum_{n=1}^{\infty}\frac{\sup_{|x|\leq n}|u(x)|}{2^{n}},\quad  u\in\mathcal{C}.
%d(u,v)=\Sigma_{n=1}^{\infty}\frac{\sup_{|x|\leq n}|u(x)-v(x)|}{2^{n}}\quad \forall u,v\in\mathcal{C}.
\end{equation}
Define $\mathcal{C}_+:=BC(\mathbb{R};\mathbb{R})$.  For $\phi,\psi\in \mathcal{C}$ we write $\phi\ge \psi$ provide that $\phi-\psi\in\mathcal{C}_+$. For the positive number $r$, we define $\mathcal{C}_r = \{v\in\mathcal{C}: 0 \leq v \leq r\}$. Given $y\in\mathbb{R}$, we define the translation operator $T_{y}$ by $T_{y}[v](x)=v(x-y)$. 

\begin{lemma}\label{SScondition}
Let $u^{\ast}$ be the minimal positive fixed point defined in Theorem \eqref{kinetic-dynamics-thm}. Assume that $h(u)$ is nondecreasing in $u\in[0,u^*]$. Assume that (A1)-(A5) hold and $L>1$. Then $Q:\mathcal{C}_{u^{\ast}}\rightarrow\mathcal{C}_{u^{\ast}}$ has the following five properties:
\begin{enumerate}
\item[(i)]$T_{y}Q = QT_{y}, y\in\mathbb{R}$; 
\item[(ii)]$Q$ is continuous with respect to the compact open topology; 
\item[(iii)]$Q$ is order preserved in the sense that $Q[u]\geq Q[v]$ whenever $u\geq v$ in $\mathcal{C}_{u^{\ast}}$; 
\item[(iv)]$\overline{Q}:[0,u^{\ast}]\to[0,u^{\ast}]$ admits two fixed points $0$ and $u^{\ast}$, and for any $\gamma\in(0,u^{\ast})$ one has $Q[\gamma]>\gamma$.
\item[(v)]$Q[\lambda\phi]\ge \lambda Q[\phi], \phi\in \mathcal{C}_{u^*}, \lambda\in(0,1)$.
\end{enumerate}
\end{lemma}
\begin{proof}
Item (i) is obvious. Item (iii) follows from the monotonicity of $h$ on $[0,u^*]$. Item (iv) follows from Theorem \eqref{kinetic-dynamics-thm}. Item (v) follows from (A5). It then remains to check item (ii), that is, we need to check $Q[\phi_n]\to Q[\phi]$ as $\phi_n\to\phi$ in $\mathcal{C}_{u^\ast} $. In virtue of (A5), we have
\[
\begin{split}
&|R(s,k_M(s-\tau(s),0,\cdot)\ast \phi_n)(x)- R(s,k_M(s-\tau(s),0,\cdot)\ast\phi)(x)| \\  &\leq \partial_\varphi R(s,0)k_M(s-\tau(s),0,\cdot)\ast f_{\sigma(s)}\ast |\phi_n-\phi|(x).
\end{split}
\]
Hence, 
\begin{equation}
|Q[\phi_n](x) - Q[\phi](x)|\le k_M(T,0,\cdot)\ast|\phi_n- \phi|(x)+ K \ast |\phi_n- \phi|(x),
\end{equation}
where $K$ is defined by
\[
K(x) := \int_{t_\alpha}^{t_\beta} \partial_\varphi R(s,0) k_M(T,s,\cdot)\ast f_{\sigma(s)}\ast k_M(s-\tau(s),0,\cdot)(x)ds.
\]
It is not difficult to see that $K\in L^1$. Therefore, 
\begin{eqnarray}
&&\|Q[\phi_n] - Q[\phi]\|\nonumber\\
=&&\sum_{k=1}^\infty 2^{-k} \sup_{|x|\le k} |Q[\phi_n](x) - Q[\phi](x)|\nonumber\\
\le && \sum_{k=1}^\infty 2^{-k} \sup_{|x|\le k}  [k_M(T,0,\cdot)+K]\ast|\phi_n- \phi|(x)\nonumber\\
=&& \sum_{k=1}^\infty 2^{-k} \sup_{|x|\le k}  \left\{\left(\int_{|y|\ge C}+\int_{|y|\le C}\right) [k_M(T,0,y)+K(y)]|\phi_n- \phi|(x-y)dy\right\}\nonumber\\
\le && 2u^* \int_{|y|\ge C}[k_M(T,0,y)+K(y)]dy \nonumber\\
&&+ \sum_{k=1}^\infty 2^{-k}\sup_{|x|\le k} \int_{|y|\le C} [k_M(T,0,y)+K(y)]|\phi_n- \phi|(x-y)dy\nonumber\\
\le && 2u^* \int_{|y|\ge C}[k_M(T,0,y)+K(y)]dy + 2C\sum_{k=1}^\infty 2^{-k} \sup_{|x|\le k+C} |\phi_n- \phi|(x)\nonumber\\
=&& 2u^* \int_{|y|\ge C}[k_M(T,0,y)+K(y)]dy + 2^{C+1}C \|\phi_n-\phi\| ,\quad C>0.
\end{eqnarray}
Consequently, 
\begin{equation}
\limsup_{n\to\infty}\|Q[\phi_n]-Q[\phi]\|\le  2u^* \int_{|y|\ge C}[k_M(T,0,y)+K(y)]dy, \quad C>0.
\end{equation}
Since $C$ is arbitrary and $k_M(T,0,\cdot)+K\in L^1$, we obtain the continuity in item (ii). 
\end{proof}
%%%%%%%%%%%%%End proof of the lemma

Now we are ready to apply Theorems 2.11, 2.15 and 2.16 and Corollary 2.16 in \cite{LiangZhao07} to get the following result on the existence of spreading speed $c^*$ and its variational formula. 
\begin{theorem}\label{TheoremSpeedOfMonotoneMap}
Assume that all conditions assumed in Lemma \ref{SScondition} are satisfied.  Define
\begin{eqnarray}\label{cstar}
  c^\ast := \inf_{\mu>0} \Phi(\mu):= \inf_{\mu>0}\frac{1}{\mu} \ln\left\{ \int_{\mathbb{R}} e^{\mu y} \mathcal{K}(y)dy \right\},
\end{eqnarray}
where $\mathcal{K}$ is defined by
\begin{equation}\label{EqMathcal-K}
\mathcal{K}(x) := k_M(T,0,x)+ \int_{t_\alpha}^{t_\beta} \partial_\varphi R(s,0) k_M(T,s,\cdot)\ast k_I(s,s-\tau(s),\cdot)\ast k_M(s-\tau(s),0,\cdot)(x)ds.
\end{equation}
Then $c^*>0$ and $c^*=\Phi(\mu^*)$ for some $\mu^*>0$. And the following statements are valid:
\begin{enumerate}
  \item[(i)] if $\varphi\in\mathcal{C}_{u^{\ast}}$ has compact support, then $\lim_{n\rightarrow\infty}\sup_{|x|\geq cn} Q^n[\varphi](x) = 0$ for $c>c^{\ast}$. 
\item[(ii)] if  $\varphi\in\mathcal{C}_{u^{\ast}}$ and $\varphi\not\equiv 0$, then $\lim_{n\rightarrow\infty}\inf_{|x|\leq cn} |Q^n[\varphi](x)-u^{\ast}| = 0$ for $c\in (0,c^{\ast})$.
\end{enumerate}
\end{theorem}
\begin{proof}
According to Lemma \ref{SScondition}, we know that $Q: \mathcal{C}_{u^*}\to \mathcal{C}_{u^*}$ satisfies all the conditions in \cite[Theorems 2.11, 2.15 and 2.16, and Corollary 2.16]{LiangZhao07}. Therefore, we have reached the conclusion except for that $c^*>0$  and it can be achieved at some $\mu^*$. Indeed, since $\mathcal{K}$ is symmetric and decays super exponentially, it then follows from the expansion $e^x=\sum_{n=0}^\infty \f{x^n}{n!}$ and Fubini's theorem that 
\begin{equation}
\int_\R  e^{\mu y} \mathcal{K}(y)dy =\int_\R \sum_{n=0}^\infty \f{\mu^ny^n}{n!}  \mathcal{K}(y)dy=\sum_{n=0}^\infty \int_\R \f{\mu^ny^n}{n!}  \mathcal{K}(y)dy=\sum_{n=0}^\infty \int_\R \f{\mu^{2n}y^{2n}}{(2n)!}  \mathcal{K}(y)dy.
\end{equation}
\begin{equation}
\int_\R  e^{\mu y} \mathcal{K}(y)dy > 1+\f{\mu^2}{2} \int_\R y^{2}\mathcal{K}(y)dy>1,
\end{equation}
which implies that $ \Phi(\mu)>0$ for $\mu>0$. By  \cite[Lemma 3.8]{LiangZhao07}, it then suffices to check $\lim_{\mu\to+\infty}\Phi(\mu)=+\infty$. Note that
\begin{equation}
\liminf_{\mu\to+\infty} \Phi(\mu)\ge \lim_{\mu\to+\infty} \ln  \left(1+\f{\mu^2}{2} \int_\R y^{2}\mathcal{K}(y)dy\right)^{\f{1}{\mu}}= \lim_{\mu\to+\infty}\f{\mu}{2}  \int_\R y^{2}\mathcal{K}(y)dy=+\infty.
\end{equation}
The proof is complete.
\end{proof}

As a consequence of the invasion speed, we see from \cite[Theorem 4.1]{LiangZhao07} that any possible speed of traveling waves is not less than the spreading speed $c^*$. 

Next we show that there exists traveling wave with $c\ge c^*$, i.e., the spreading speed $c^*$ is also the minimal wave speed. Indeed,  \cite[Theorem 4.2]{LiangZhao07} applies to the iterative system $\{Q^n\}_{n\ge 1}$ if $Q$ is compact. However, as explained before, when the mature population does not move, then the result mapping $Q$ is not compact. To include such a case, instead of  \cite[Theorem 4.2]{LiangZhao07} we shall employ an improved version with weaker compactness assumptions, for which we refer to  \cite[Theorem 3.8]{FangZhao14}. In particular, for the standing mapping $Q$, we do not need to impose any compactness condition, but the phase space is chosen to be the monotone function space
\begin{equation}
\mathcal{M}:=\{\phi:\R\to\R| \phi(x)\ge \phi(y), x\le y\},
\end{equation}
which is endowed with the compact open topology.  Similarly, we may define the ordering in $\mathcal{M}$ and the order interval $\mathcal{M}_{u^*}$.

\begin{theorem}\label{TW}
Assume that all the conditions in Lemma \ref{SScondition} hold. Then the spreading speed $c^*$ is also the minimal speed of traveling waves connecting $0$ to $u^*$. 
\end{theorem}
\begin{proof}
It is not difficult to check that the first four item of conclusions in Lemma \ref{SScondition} still hold with $\mathcal{C}_{u^*}$ being replaced by $\mathcal{M}_{u^*}$. Next we check that $Q:\mathcal{M}_{u^{\ast}}\to \mathcal{M}_{u^{\ast}}$ maps left continuous function to left continuous functions. Indeed, for $\varphi\in\mathcal{M}_{u^{\ast}}$ and $t>s\geq0$, if $\varphi$ is left continuous, then $k_M(t,s,\cdot)\ast\varphi$ is left continuous, so is $Q[\varphi]$. Then we can employ \cite[Theorem 3.8 and Remark 3.7]{FangZhao14} to obtain that $c^*$ is the minimal wave speed.
\end{proof}

In order to study the parameter influence on $c^*$, we use the explicit form of Green's function $k_M(t,s,x)$ of  $\partial_t\rho=D_M(t)\partial_{xx}\rho-d_M(t)\rho$ to compute the integral $\int_\R e^{\mu y}\mc{K}(y)dy$, where $\mc{K}$ is defined in \eqref{EqMathcal-K}. Indeed, it is known that 
\[
k_M(t,s,x)=\frac{1} {\sqrt{4\pi\int^{t}_{s}D_M(\varsigma)d\varsigma}} \exp\left\{-\frac{x^{2}}{4\int^{t}_{s}D_M(\varsigma)d\varsigma}-\int^{t}_{s}d_M(\varsigma) d\varsigma\right\}.
\]
Note that
\[
\int_\R e^{\mu y} k_M(t,s,y)dy= \exp\left\{  \int_s^t [\mu^2D_M(\varsigma)-d_M(\varsigma)] d\varsigma \right\}.
\]
By the expressions of $\mc{K}$ defined in \eqref{EqMathcal-K}, we compute to have
\begin{eqnarray}\label{computation}
\int_\R e^{\mu y}\mc{K}(y)dy
=&& e^{\int_0^T[\mu^2D_M(\varsigma)-d_M(\varsigma)] d\varsigma}\nonumber\\
&&+\int_{t_\alpha}^{t_\beta}\partial_{\varphi}R(s,0) e^{\left(\int_s^T+\int_0^{s-\tau(s)}\right) [\mu^2D_M(\varsigma)-d_M(\varsigma)] d\varsigma+\mu^2\sigma(s)}ds.
\end{eqnarray}
This formula will be used several times in the rest of this section.

\subsection{Influence of periodic delay $\tau(t)$ on $c^*$} In this subsection, we compare the spreading speeds subject to two different maturation periods; one depends on time, the other is its average.  For this purpose, we define 
\begin{equation}\label{EqAverageOfTau(s)}
\tau_{av}:=\f{1}{t_\beta - t_\alpha}\int_{t_\alpha}^{t_\beta} \tau(s)ds.
\end{equation}
To focus on the influence of time periodicity of delay, we assume all other time periodic functions are trivial.  As such, $D_M(s)\equiv D_M, D_I(s)\equiv D_I$ and 
\[
\partial_{\varphi}R(s,0)=(1-\tau'(s))e^{-d_I \tau(s)} ph'(0).
\] 
Now we have two spreading speeds: $c^*(\tau)$ and $c^*(\tau_{av})$. To compare them, it suffices to compare the following two integrals thanks to the variational formula of $c^*$ in \eqref{cstar} and \eqref{computation}:
\begin{equation}\label{twoterm}
\int_{t_{\alpha}}^{t_{\beta}} (1-\tau'(s)) e^{l(\mu)\tau(s)}ds\quad\text{and}\quad \int_{t_{\alpha}}^{t_{\beta}} e^{l(\mu)\tau_{av}}ds,
\end{equation}
where $l(\mu)=(d_M-d_I)+\mu^{2}(D_I-D_M)$. Next, by Jensen's inequality and the fact $1-\tau'(s)>0$ as assumed in (B5), we have
\begin{equation}\label{ineq1}
\ln\left(\f{1}{t_\beta-t_\alpha}\int_{t_{\alpha}}^{t_{\beta}} (1-\tau'(s)) e^{l(\mu)\tau(s)}ds\right)\le \f{1}{t_\beta-t_\alpha}\int_{t_{\alpha}}^{t_{\beta}} \ln \left(  (1-\tau'(s)) e^{l(\mu)\tau(s)} \right)ds.
\end{equation}
Further, since
\begin{equation}\label{ineq2}
\ln \left(  (1-\tau'(s)) e^{l(\mu)\tau(s)} \right)=\ln (1-\tau'(s))+l(\mu)\tau(s),
\end{equation}
it then follows that
\begin{equation}
\int_{t_{\alpha}}^{t_{\beta}} (1-\tau'(s)) e^{l(\mu)\tau(s)}ds\le \int_{t_{\alpha}}^{t_{\beta}} e^{l(\mu)\tau_{av}}ds
\end{equation}
provided that
\begin{equation}\label{ineq3}
\int_{t_{\alpha}}^{t_{\beta}}  \ln (1-\tau'(s))ds\le 0.
\end{equation}
Using the inequality $\ln x\le x-1, x>0$ we infer that \eqref{ineq3} holds if 
\begin{equation}
\int_{t_{\alpha}}^{t_{\beta}}-\tau'(s)ds\le 0,
\end{equation}
that is,
\begin{equation}\label{ineq4}
\tau(t_\alpha)\le \tau(t_\beta).
\end{equation}
Thus, $c^*(\tau)\le c^*(\tau_{av})$ provided that \eqref{ineq4} holds.

On the other hand, to expect $c^*(\tau)\ge c^*(\tau_{av})$ one has to assume that $\tau(s)$ is decreasing for $s$ in some intervals. For this purpose, we consider the case where $\tau(s)$ is a nonincreasing linear function, say
\begin{equation}
\tau(s):=\theta_1-\theta_2 s,\quad \theta_2\ge 0, s\in [t_\alpha,t_\beta].
\end{equation}
As such, the first term in \eqref{twoterm} can be estimated from below as follows:
\begin{eqnarray}
\int_{t_{\alpha}}^{t_{\beta}} (1-\tau'(s)) e^{l(\mu)\tau(s)}ds && \ge (1+\theta_2) \int_{t_{\alpha}}^{t_{\beta}} e^{l(\mu)(\theta_1-\theta_2 s)}ds\nonumber\\
&& =\f{1+\theta_2}{l(\mu) \theta_2}e^{l(\mu)\theta_1}[e^{-l(\mu)\theta_2t_\alpha}- e^{-l(\mu)\theta_2t_\beta}]
\end{eqnarray}
Meanwhile, we compute the second term in \eqref{twoterm} to have
\begin{equation}
 \int_{t_{\alpha}}^{t_{\beta}} e^{l(\mu)\tau_{av}}ds=(t_\beta-t_\alpha) e^{l(\mu)[\theta_1-\theta_2\f{t_\alpha+t_\beta}{2}]}.
\end{equation}
Set $\zeta:=\f{1}{2}l(\mu) \theta_2(t_\beta-t_\alpha)$. Then we arrive at 
\begin{equation}
\int_{t_{\alpha}}^{t_{\beta}} (1-\tau'(s)) e^{l(\mu)\tau(s)}ds\ge \int_{t_{\alpha}}^{t_{\beta}} e^{l(\mu)\tau_{av}}ds
\end{equation}
provided that 
\begin{equation}
\f{1+\theta_2}{2}\ge \f{\zeta}{e^{\zeta}-e^{-\zeta}}, \quad \zeta>0,
\end{equation}
which is true because $ \f{\zeta}{e^{\zeta}-e^{-\zeta}}$ is increasing in $\zeta>0$ and has limit $\f{1}{2}$ as $\zeta\to 0^+$.

We summarize the above analysis to have the following result.
\begin{proposition}\label{p1}
(i) If $\tau(t_\alpha)<\tau(t_\beta)$, then $c^*(\tau)< c^*(\tau_{av})$; (ii) If $\tau(s)=\theta_1-\theta_2 s$ with $\theta_2> 0$, then $c^*(\tau)>c^*(\tau_{av})$.
\end{proposition}

{\bf Biological interpretation}: The condition $\tau(t_\alpha)<\tau(t_\beta)$ means that more maturation time is needed at the end of the breeding season than the beginning. This appeals to be a common biological scenario, in which the time heterogeneity of development can slow down the invasion. In other words, using the average of maturation time to compute the invasion speed is overestimated.

\subsection{Dependence of $c^*$ on the death rate $d_M(t)$} Assume that $d_M(t)$ consists of two parts; one is the intrinsic death rate $d(t)$, the other is the extrinsic death rate $\eta(t)$. For $C>0$, define 
\[
\Omega: =\{\eta \geq 0: \eta \in BC([0,T], \R^+),\eta(0)=\eta(T), \f{1}{T}\int_{0}^{T}\eta (s)ds=C\}.
\] 
We plan to minimize $c^*$ subject to the following constraints 
\begin{equation}\label{constraints}
\f{1}{T}\int_{0}^{T}\eta (s)ds=C.
\end{equation}
For this purpose, we write $c^*(\eta)$ as a functional of $\eta\in \Omega$. Then we have the following result.
\begin{proposition}\label{p2}
$\min_{\eta\in \Omega }c^{\ast}(\eta)$ is achieved at $\eta=\eta^*$ if and only if $\eta^*(s)\equiv 0, s\in [\alpha, t_\beta]$.
\end{proposition}
\begin{proof}
Using \eqref{computation} and \eqref{constraints}, we compute to have
\begin{eqnarray}
&&\int_\R e^{\mu y}\mc{K}(y)dy\nonumber\\
=&& e^{\int_0^T[\mu^2D_M(\varsigma)-d(\varsigma)] d\varsigma -TC} \nonumber\\
&&+\int_{t_\alpha}^{t_\beta}\partial_{\varphi}R(s,0) e^{\left(\int_s^T+\int_0^{s-\tau(s)}\right) [\mu^2D_M(\varsigma)-d(\varsigma)] d\varsigma-TC+\int_{s-\tau(s)}^s \eta(\varsigma)d \varsigma+\mu^2\sigma(s)}ds\nonumber\\
:=&& F(\eta),
\end{eqnarray}
from which we clearly see that the functional $F(\eta)$, subject to the constraints \eqref{constraints} is non-increasing in $\eta$ and independent of $\eta(s), s\not\in [\alpha, t_\beta]$. Combining the definition of $c^*$ in \eqref{cstar}, we infer that $c^*(\eta)$ is minimized if $\eta(s)\equiv 0, s\in [\alpha, t_\beta]$.
\end{proof}

To finish this subsection, we remark that from the above proof one may find that a similar result holds if the diffusion rate $D_M(t)=D(t)+\eta(t)$ is chosen to be the parameter. 

\subsection{Increasing order of $c^*$ on the diffusion rate $D_M(t)$} Assume that $D_M(t)=kD(t), k>0$, where $D(t)$ is a positive $T$ periodic function. Next we investigate the asymptotic behavior $c^{\ast}$ as $k\rightarrow +\infty$. For this purpose, we write $c^*=c^*(k)$.

\begin{proposition}\label{p3}
$\lim_{k\rightarrow+\infty}\frac{c^{\ast}(k)}{\sqrt{k}}=\inf_{\mu>0}H(\mu,+\infty) \in (0,+\infty)$, where
  \begin{equation}\nonumber
\begin{split}
  H(\mu,+\infty):=&\frac{1}{\mu}\ln\left(e^{\mu^{2}\int_{0}^{T}D(s)ds-\int_{0}^{T}d(s)ds}
  +\int_{t_{\alpha}}^{t_{\beta}}\partial_\varphi R(s,0) \right.\\ &\left.e^{-\left(\int_{0}^{s-\tau(s)}+\int_{s}^{T}\right)d(\omega)d\omega}e^{\mu^{2} \left(\int_{0}^{s-\tau(s)}+\int_{s}^{T}\right)D(\omega)d\omega} ds             \right).
\end{split}
\end{equation}
\end{proposition}
\begin{proof}
By \eqref{computation}, we have
\begin{eqnarray*}
\int_\R e^{\mu y}\mc{K}(y)dy
=&& e^{\int_0^T[\mu^2kD(\varsigma)-d_M(\varsigma)] d\varsigma}\nonumber\\
&&+\int_{t_\alpha}^{t_\beta}\partial_{\varphi}R(s,0) e^{\left(\int_s^T+\int_0^{s-\tau(s)}\right) [\mu^2kD(\varsigma)-d_M(\varsigma)] d\varsigma+\mu^2\sigma(s)}ds.
\end{eqnarray*}
Introducing the variable change $\nu=\sqrt{k}\mu$, we then have
\begin{equation}\nonumber
\begin{split}
  \frac{c^{\ast}(k)}{\sqrt{k}}=&\inf_{\nu>0} \frac{1}{\nu}\ln\left(e^{\nu^{2}\int_{0}^{T}D(s)ds-\int_{0}^{T}d(s)ds}
  +\int_{t_{\alpha}}^{t_{\beta}} \partial_\varphi R(s,0) \right.\\& \left.e^{\left(\int_s^T+\int_0^{s-\tau(s)}\right) [\nu^2D(\varsigma)-d_M(\varsigma)] d\varsigma+\f{\nu^2}{k}\sigma(s)}ds
          \right)\\
          :=& \inf_{\nu>0}H(\nu, k).
\end{split}
\end{equation}
Since $H(\nu,k)$ is strictly decreasing in $k$ and 
\[
\lim_{\mu\rightarrow 0^{+}}H(\mu,k)=+\infty=\lim_{\mu\rightarrow +\infty}H(\mu,k), \quad k\in (0,+\infty],
\]
$\inf_{\mu>0}H(\mu,k)$ can be archived at $\nu=\nu(k)$. Consequently, $H(\nu(k),k)$ is non-increasing in $k$. Note that $H(\nu, +\infty)>0$ exists and $\inf_{\nu>0}H(\nu, +\infty)$ can be archived at some finite $\nu$. It then follows that 
\[
H(\nu(k), k)\ge H(\nu(k), +\infty)\ge \inf_{\nu>0}H(\nu, +\infty), \quad k>0.
\] 
Hence, the limit $\lim_{k\rightarrow+\infty}\frac{c^{\ast}(k)}{\sqrt{k}}$ exists and it is not less than the positive number $\inf_{\nu>0}H(\nu, +\infty)$. 

Next, we show that the limit equals  $\inf_{\nu>0}H(\nu, +\infty)$. Indeed, Assume $\inf_{\nu>0}H(\nu,\infty)$ is archived at $\nu=\nu_\infty$. Then $\lim_{k\rightarrow\infty}H(\nu_{\infty},k)=H(\nu_{\infty},\infty)$ due to the convexity of $H(\mu, k)$ in $\mu$ and the monotonicity in $k$, and consequently, 
\[
\lim_{k\rightarrow \infty}\frac{c^{\ast}(k)}{\sqrt{k}}= \lim_{k\rightarrow\infty}\inf_{\mu>0}H(\mu,k)=H(\mu_{\infty},\infty)=\inf_{\nu>0}H(\nu,+\infty).
\]

Finally, we prove $\inf_{\nu>0}H(\nu,+\infty)>0$. Indeed, it suffices to show that $H(\nu, +\infty)>0$ for $\nu>0$, which holds provided that 
\[
e^{-\int_{0}^{T}d(s)ds}+\int_{t_{\alpha}}^{t_{\beta}} \partial_\varphi R(s,0) e^{-\left(\int_{0}^{s-\tau(s)}+\int_{s}^{T}\right)d(\omega)d\omega}ds>1.
\]
This is equivalent to the assumption $L>1$ for the instability of $0$, as already assumed in Theorem \ref{kinetic-dynamics-thm}. The proof is complete. 
\end{proof}

To finish this subsection, we remark that $\inf_{\nu>0}H(\nu,+\infty)$ is the spreading speed of the extreme case where $D_I(t)\equiv 0$, i.e., the immature population do not move.

\section{Propagation dynamics of model \eqref{Model}}
 In the previous sections, we have obtained the propagation dynamics for the reduced iterative system $\{Q^n\}_{n\ge 0}$.  In particular, under appropriate conditions, there is a spreading speed $c^*$ that coincides with the minimal speed of traveling waves. Now we come back to the original reaction-diffusion model \eqref{Model}. Firstly, we employ the evolution idea introduced in \cite{LiangYiZhao06} to show that $c^*/T$ is the spreading speed and the minimal speed of time periodic traveling waves for the the mature population. Secondly, we use the conservation equality \eqref{conservation} to prove that the immature population share the same propagation dynamics.

\subsection{The mature equation}

The evolution idea in \cite{LiangYiZhao06} says that for a time $T$-periodic semiflow $\{P_t\}_{t\ge 0}$, the function $W(t,x-ct):=P_t[U](x-ct+ct)$ is a $T$-time periodic traveling wave solution provided that $P_T[U](x)=U(x-cT)$. Thus, we first use the $u$-equation of \eqref{Model}  to define a  time periodic semiflow. Generally speaking, \eqref{Model}  is a reaction-diffusion equation with time delay, one may try to choose $C([-\max_{s}\tau(s),0]\times \R, \R)$ or its subset as the phase space (see for instance \cite{Smith08}). However, the first equation of \eqref{Model}, that is,
\begin{equation}\label{u-eq}
\frac{\partial u}{\partial t}=D_M(t)\frac{\partial^{2} u}{\partial x^{2}}-d_M(t)u+R(t,u(t-\tau(t), \cdot))(x)
\end{equation}
has a special nonlinearity due to the assumption (B2). In particular, one can use $C(\R, \R)$ as the phase space because the solution can be uniquely determined by $u(0,x)$. Therefore, we can define a time periodic semiflow by 
\[
P_t[\phi]=u(t,\cdot;\phi),
\]
where $u(t,x;\phi)$ is the solution of \eqref{u-eq} with $u(0,\cdot;\phi)=\phi\in C(\R,\R)$ and $0\le \phi \le u^*$. Here $u^*$ is defined as in Theorem \ref{kinetic-dynamics-thm}.

By the reduction process in section 2.2, we know that the map $Q$ defined in \eqref{EqQ=S+R} equals $P_T$. Hence, $\bar{u}(t)=P_t[u^*]$ is a periodic solution of \eqref{u-eq}. Then, as a consequence of 
Theorems  \ref{kinetic-dynamics-thm} and \ref{TW} and \cite[Theorems 2.1-2.3]{LiangYiZhao06}, we have the following result.
\begin{theorem}\label{periodicTW}
Let $c^{\ast}$ be the spreading speed of $Q$. Then the following statements hold.
\begin{enumerate}
  \item[(i)]
  For $c\in(0, c^{\ast}/T)$, the first equation of \eqref{Model} has no $T$-time periodic traveling wave $U(t,x-ct)$ connecting $\bar{u}(t)$ to $0$.
  \item[(ii)]
  For $c\geq c^{\ast}/T$, the first equation of \eqref{Model} has a $T$-time periodic traveling wave solution $U(t,x-ct)$ connecting $\bar{u}(t)$ to $0$. Moreover, $U(t,\xi)$ is left continuous and nonincreasing in $\xi\in\mathbb{R}$.
\end{enumerate}
\end{theorem}

\subsection{The immature equation}
In this section we will study the propagation dynamics of the immature population equation, that is, 
\begin{equation}\label{v-eq}
\frac{\partial v}{\partial t}=D_{I}(t)\frac{\partial^{2} v}{\partial x^{2}}-d_{I}(t)v+b(t,u(x,t))-R(t,u(t-\tau(t), \cdot))(x),
\end{equation}
where $u$ is supposed to be known. 

We first prove a conservation equality, which biological means that all newborns will become mature in the same year as when they were born.
\begin{lemma}
Let $u(x,t)$ be a solution of equation \eqref{u-eq}. Then one has
\begin{equation}\label{conservation1}
\int_0^T k_I(t,s,\cdot)\ast b(s,u(s,\cdot))ds\equiv \int_{0}^T k_I(t,s,\cdot)\ast R(s,u(s-\tau(s),\cdot))ds, \quad t>T,
\end{equation}
where $k_I(t,s,x)$ is the Green function of $\partial_t \rho=D_I(t)\partial_{xx}\rho-d_I(t)\rho$.  If $u(s,x) (\equiv u(s))$ is independent of $x$, then \eqref{conservation1} reduces to 
\begin{equation}\label{conservation2}
\int_{0}^{T}e^{-\int^{t}_{s}d_{I}(\varsigma)d\varsigma} b(s,u(s))ds = \int_{0}^{T}e^{-\int^{t}_{s}d_{I}(\varsigma)d\varsigma} R(s,u(s-\tau(s)))ds.
\end{equation}
\end{lemma}

\begin{proof}
From the definition of $R$ as in \eqref{recruitment}, we see that
\begin{equation}
R(s,\phi)=(1-\tau'(s))k_I(s,s-\tau(s),\cdot)\ast b(s-\tau(s), \phi),
\end{equation}
which, combining with the group property of $k_I$, yields that
\[
k_I(t,s,\cdot)\ast R(s,u(s-\tau(s),\cdot))=(1-\tau'(s))k_I(t,s-\tau(s),\cdot)\ast b(s-\tau(s), u(s-\tau(s), \cdot)).
\]
Note that $R(s, \phi)\equiv 0$ for $s\not\in [t_\alpha, t_\beta]$ and $b(s, \phi)\equiv 0$ for $s\not\in [\alpha, \beta]$. It then follows that 
\begin{eqnarray}
\int_{0}^T k_I(t,s,\cdot)\ast R(s,u(s-\tau(s),\cdot))ds &&=\int_{t_\alpha}^{t_\beta} k_I(t,s,\cdot)\ast R(s,u(s-\tau(s),\cdot))ds \nonumber\\
&& =\int_\alpha^\beta k_I(t,\eta,\cdot)\ast b(\eta, u(\eta, \cdot))d\eta\nonumber\\
&&=\int_0^T  k_I(t,\eta,\cdot)\ast b(\eta, u(\eta, \cdot))d\eta,
\end{eqnarray}
where the variable change $\eta=s-\tau(s)$ is used.  
\end{proof}

Next we write the linear inhomogeneous reaction-diffusion equation \eqref{v-eq} as the following integral form and then investigate its propagation dynamics. 
\begin{equation}\label{solution-of-v}
v(x,t) = k_I(t,0,\cdot)\ast v(\cdot,0) + \int_{0}^{t} k_I(t,s,\cdot)\ast Z(s,\cdot)(x) ds,
\end{equation}
where 
\begin{equation}\label{Z}
Z(s,x):=b(s,u(s,x))-R(s,u(s-\tau(s),\cdot))(x).
\end{equation}

Now we are in a position to present the propagation dynamics of the immature population. 
\begin{theorem}\label{v-TW}
The following statements are valid:
\begin{enumerate}
\item[(i) ] If $u(t,x)\equiv \bar{u}(t)$, then \eqref{v-eq} admits a unique nontrivial bounded periodic solution $\bar{v}(t)$ that is independent of $x$.
\item[(ii)] If $u(t,x)=U(t,x-ct)$ is a periodic traveling wave, as established in Theorem \ref{periodicTW}, then  \eqref{v-eq} admits a unique periodic traveling wave $V(t,x-ct)$ with $V(t,+\infty)=0$ and $V(-\infty)=\bar{v}(t)$.
\end{enumerate}
\end{theorem}
\begin{proof}
We first prove the uniqueness. Indeed, assume for the sake of contradiction that there are two solutions $v_1(x,t),v_2(x,t)$. Then $\tilde{v}:=v_1-v_2$ satisfies $\tilde{v}_t=D_I(t)\tilde{v}_{xx}-d_I(t)\tilde{v}$, for which the only bounded solution is zero. Thus, the uniqueness is proved. 

Next we prove the existence. Indeed, choosing $v(0,x)\equiv 0$, we obtain a special solution 
\[
v(x,t)=\int_0^t k_I(t,s,\cdot)\ast Z(s,\cdot)(x) ds.
\] 
Now we proceed with the two cases. (i) $u(t,x)\equiv \bar{u}(t)$. Note that $Z(x,s)$ is assumed to be independent of $x$, so is $v(x,t)$. Thus,  we may write $\bar{v}(t)$ and $\bar{Z}(t)$ instead of $v(x,t)$ and $Z(x,t)$, respectively. Note that $\bar{Z}$ is periodic with $\int_0^T e^{-\int^{t}_{s}d_{I}(\varsigma)d\varsigma} \bar{Z}(s)ds=0$ in virtue of \eqref{conservation2}. It then follows that 
\begin{eqnarray}
  \bar{v}(t+T)&&= \int_{0}^{t+T} e^{-\int^{t+T}_{s}d_{I}(\varsigma)d\varsigma} \bar{Z}(s)ds\nonumber\\
  &&=\int_{0}^{T}e^{-\int^{t+T}_{s}d_{I}(\varsigma)d\varsigma} \bar{Z}(s) ds+\int_{T}^{T+t}e^{-\int^{t+T}_{s}d_{I}(\varsigma)d\varsigma} \bar{Z}(s) ds\nonumber\\
&&=e^{-\int_t^{t+T}d_{I}(\varsigma)d\varsigma }\int_{0}^{T}e^{-\int^{t}_{s}d_{I}(\varsigma)d\varsigma} \bar{Z}(s) ds +\int_{0}^{t}e^{-\int^{t+T}_{s+T}d_{I}(\varsigma)d\varsigma} \bar{Z}(s+T) ds\nonumber\\
&&=\int_{0}^{t}e^{-\int^{t}_{s}d_{I}(\varsigma)d\varsigma} \bar{Z}(s) ds\nonumber\\
&&=\bar{v}(t),
 \end{eqnarray}
 where the periodicity of $d_I$ is also used. (ii) $u(t,x)=U(t,x-ct)$. In this case, $Z(s,x)=b(s, U(s,x-cs))-R(s,U(s-\tau(s), \cdot-cs+c\tau(s)))(x)$. Define
 \begin{equation}\label{def-V}
 V(t,\xi):=\int_0^t k_I(t,s,\cdot)\ast Z(s,\cdot)(\xi+ct) ds.
 \end{equation}
Obviously, $V(t,x-ct)$ is a solution of \eqref{v-eq} with zero initial value. Then we prove that $V$ is periodic in $t$ with $V(t,+\infty)=0$ and $V(t, -\infty)=\bar{v}(t)$. Indeed, notice that 
\begin{equation}
Z(s+T, x+cT)=Z(s,x)
\end{equation}
and there exists a constant $l_{T}>0$ such that
\begin{equation}
k_I(t+T, s, x)=l_T k_I(t,s,x).
\end{equation}
It then follows from \eqref{conservation1} that 
\begin{eqnarray}
 V(t+T,\xi)&&= \left(\int_0^T +\int_T^{t+T}\right ) k_I(t+T,s,\cdot)\ast Z(s,\cdot)(\xi+ct+cT) ds\nonumber\\
 &&=\int_T^{t+T} k_I(t+T,s,\cdot)\ast Z(s,\cdot)(\xi+ct+cT) ds\nonumber\\
 &&=\int_0^{t} k_I(t+T,s+T,\cdot)\ast Z(s+T,\cdot)(\xi+ct+cT) ds\nonumber\\
 &&=\int_0^{t} k_I(t+T,s+T,\cdot)\ast Z(s+T,\cdot+cT)(\xi+ct) ds\nonumber\\
 &&=V(t,\xi).
\end{eqnarray}
Finally, we prove the limits. Indeed, since
\begin{equation}
Z(s,\pm \infty)=b(s, U(s,\pm\infty))-R(s, U(s-\tau(s), \pm\infty))
\end{equation} 
uniformly in $s\in \R$ due to the periodicity in $s$. Then in \eqref{def-V}, passing $\xi\to\pm \infty$ in advantage of the Lebesgue dominated convergence theorem and the periodicity of $V(t,\xi)$ in $t$, we obtain
\begin{equation}
 V(t,\pm\infty)=\int_0^t   e^{-\int^{t}_{s}d_{I}(\varsigma)d\varsigma}  Z(s,\pm\infty)ds
\end{equation}
uniformly for $t\in\R$.  Clearly, $Z(s,+\infty)=0$, so is $V(t,+\infty)$. Since $Z(s,-\infty)=b(s,\bar{u}(s))-R(s,\bar{u}(s-\tau(s)))$, we see from (i) that $V(t,-\infty)=\bar{v}(t)$.  
\end{proof}

\begin{theorem}
Let $u(x,t)$ be a solution of \eqref{u-eq} with spreading speed $c^*/T$. Then for any bounded initial value, the solution of \eqref{v-eq} propagates asymptotically with speed $c^*/T$. More precisely,  $\lim_{t\to +\infty} \sup_{|x|\geq ct}v(x,t)=0$ for $c>c^{\ast}/T$ and $\lim_{t\to+\infty}\sup_{|x|\leq ct}|v(x,t)-\bar{v}(t)|=0$ for $c\in(0,c^{\ast}/T)$.
\end{theorem}
\begin{proof}
We first claim that for any $s\in [0,t)$ and $y\in\R$, $Z(t-s,x-y)$ propagates asymptotically with speed $c^*/T$. Let us postpone the poof of the claim and quickly reach the conclusion. Define $M:=\sup_{x\in\R}\|v(x,0)\|$. By \eqref{solution-of-v}, we have
\begin{equation}
v(x,t)\le M e^{-\int^{t}_{0}d_{I}(\varsigma)d\varsigma} +\int_0^t \int_\R k_I(t,t-s,y)Z(t-s,x-y)dyds.
\end{equation}
Using Lebesgue's dominated convergence theorem, we obtain $\lim_{t\to\infty}\sup_{|x|\ge ct}v(x,t)=0$ for $c>c^*/T$. From the proof of Theorem \ref{v-TW} we know that 
\[
\bar{v}(t)=\int_0^t e^{-\int^{t}_{s}d_{I}(\varsigma)d\varsigma} Z(s,-\infty)ds.
\]
Consequently, 
\begin{eqnarray}
|v(x,t)-\bar{v}(t)| \le &&  M e^{-\int^{t}_{0}d_{I}(\varsigma)d\varsigma} + \left|\int_0^t \int_\R k_I(t,s,y) Z(s, x-y)dyds\right. \nonumber\\
&&\left. -\int_0^t e^{-\int^{t}_{s}d_{I}(\varsigma)d\varsigma} Z(s,-\infty)ds\right|\nonumber\\
\le && M e^{-\int^{t}_{0}d_{I}(\varsigma)d\varsigma} + \int_0^t \int_\R k_I(t,s,y)| Z(s, x-y)-Z(s,-\infty)|dyds\nonumber\\
=&& M e^{-\int^{t}_{0}d_{I}(\varsigma)d\varsigma} + \int_0^t \int_\R k_I(t,t-s,y)| Z(t-s, x-y)-Z(t-s,-\infty)|dyds,\nonumber\\
\end{eqnarray}
which implies that $\lim_{t\to\infty}\sup_{|x|\le ct} |v(x,t)-\bar{v}(t)| =0$ uniformly in $t\in\R$ for $c\in(0,c^*/T)$, thanks to Lebesgue's dominated convergence theorem.

Proof of the claim. Fix $s$ and $y$. By the same arguments as in the proof of \cite[Theorem 3.2]{Fang-Wei-Zhao2008}, we can infer that $Z(t-s,x-y)$ propagates asymptotically with speed $c^*/T$ provided that $Z(t,x)$ propagates with the same speed. Indeed, since the birth function $b$ is sublinear, there holds
\begin{equation}
Z(t,x)\le b(t, u(t,x))\le b'(t, 0)u(t,x)
\end{equation}
Thus, $\lim_{t\to\infty}\sup_{|x|\ge ct} Z(t,x)=0$ for $c>c^*/T$. On the other side, since the birth function $b$ is Lipschitz continuous, there exists $C>0$ such that for $c\in (0, c^*/T)$,
\begin{eqnarray*}
&&\lim_{t\to\infty}\sup_{|x|\le ct}\left|Z(t,x)-[b(t,\bar{u}(t))-R(t,\bar{u}(t-\tau(t)))]\right|\\
\le && C\lim_{t\to\infty}\sup_{|x|\le ct} [|u(t,x)-\bar{u}(t)|+|u(t-\tau(t),\cdot)-\bar{u}(t-\tau(t))|\ast f_{\sigma(t)}].
\end{eqnarray*}
Note that, as assumed, the first term has limit zero. The second also has limit zero thanks to the same arguments as in the proof of \cite[Theorem 3.2]{FangWeiZhao08}. 
\end{proof}

\section{Summary and Discussion}

A time periodic and diffusive model system in unbounded domain is proposed to study the seasonal influence on the propagation dynamics of a single invasive species. The two-stage structure by age and the seasonal developmental rate jointly result in a time periodic delay, which, combining with the mobility of immature population, then gives rise to the spatial non-locality for the recruitment of mature population. Further, the scenario of distinct breading and maturation seasons within the same year makes the mathematical analysis presented in this paper possible. Indeed, it is the key to reduce the model system to an explicit mapping for the mature population, partially coupled by a linear and inhomogeneous diffusive equation for the immature population. Finally, some recently developed dynamical system theories apply to the mapping under suitable technical assumptions, and the finding of the conservation equality plays a vital role in the study of the linear inhomogeneous equation once the reduced mapping is well understood. 

In Theorems \ref{TheoremSpeedOfMonotoneMap} and \ref{TW}, we established the spreading speed $c^*$ and its coincidence with the minimal wave speed for model \eqref{Model}. Some seasonal influences on $c^*$ are also obtained. In particular, it is known from literature \cite{LiRuanWang07, WangLiRuan08} that time delay decreases the speed, and it is shown by Proposition \ref{p1} that the seasonality can further decrease the speed if the development rate is decreasing in time during the maturation season. Proposition \ref{p2} shows that the extrinsic death rate in the season without juveniles contributes more than other seasons to decrease the speed, while the consequent remark suggests the opposite for extrinsic diffusion rate.  Proposition \ref{p3} shows that the speed is asymptotic to infinity with the same order as the square root of the diffusion rate as it increases to infinity. 

In a word, a useful message for the optimal control of the biological invasion with yearly generation structure is to kill the mature population or restrict their mobility in the season without juveniles. 

In this paper, the proposed model \eqref{Model} with a special scenario under suitable technical conditions, including the monotonicity and sublinearity,  is analyzed. By considering other biological scenarios, one may have several interesting and challenging mathematical questions. For instance, could the invasion be sped up if 
the diffusion mechanism is modeled by the nonlocal dispersal? Is there any new dynamics if the monotonicity condition imposed in this paper is invalid? What if a weak or strong Allee effect in the breeding season is introduced? How to incorporate spatial heterogeneity into the model and how does it influence the dynamics? These questions are under the authors' investigations.

\section*{Acknowledgments} This work received fundings from the National Natural Science of Foundation of China (No. 11371111, 11771108) and the Fundamental Research Funds for the Central Universities of China.

%-----------------------------------------------
%% References
%-----------------------------------------------

\end{document}